\documentclass[10pt]{amsart}  
\usepackage{Paper_Style}

\begin{document}
\author{Marco Aldi}
\author{Jeffrey Buffkin II}
\author{Cody Cline}
\author{Sean Cox}

\title{Sparse systems of functions and Quasi-analytic classes}

\thanks{Partially funded by VCU SEED Grant and NSF grant DMS-2154141 of Cox.}


\begin{abstract}
We provide a new characterization of quasi-analyticity of Denjoy-Carleman classes, related to \emph{Wetzel's Problem}.  We also completely resolve which Denjoy-Carleman classes carry \emph{sparse systems}:  if the Continuum Hypothesis (CH) holds, \textbf{all} Denjoy-Carleman classes carry sparse systems; but if CH fails, a Denjoy-Carleman class carries a sparse system if and only if it is not quasi-analytic.  As corollaries, we extend results of \cite{MR3552748} and \cite{CodyCoxLee} about non-existence of ``anonymous predictors" for real functions.
\end{abstract}

\maketitle

{
  \hypersetup{linkcolor=black}
  \tableofcontents
}

\section{Introduction}

Motivated by some questions about ``anonymous predictors" of Hardin-Taylor~\cite{MR3100500} and Bajpai-Velleman~\cite{MR3552748} and an old theorem of Erd\H{o}s~\cite{MR168482} about \emph{Wetzel's Problem}, Cody-Cox-Lee~\cite{CodyCoxLee} introduced the following concept.  In what follows, $\hm$ refers to the set of increasing bijections from $\mathbb{R}$ to $\mathbb{R}$, and if $P \in \mathbb{R}^2$, we use $x_P$ and $y_P$ to denote the first and second coordinate of $P$, respectively.

\begin{definition}\label{def_SparseSystem}
Suppose $\Gamma \subset \hm$.  A \textbf{sparse $\boldsymbol{\Gamma}$-system} is a collection of functions
\[
\left\{ f_P \ : \ P \in \mathbb{R}^2\right\}
\]
such that:
\begin{enumerate}
    \item Each $f_P$ is a member of $\Gamma$ that passes through the point $P$; and
    \item For every $u \in \mathbb{R}$, both 
    \[
    \left\{ f_P(u) \ : \ P \in \mathbb{R}^2 \text{ and } u \ne x_P \right\}
    \]
    and
    \[
    \left\{ f^{-1}_P(u) \ : \ P \in \mathbb{R}^2 \text{ and } u \ne y_P \right\}
    \]
    are countable.
\end{enumerate}
If \hspace{.2mm}$\Gamma$ is not necessarily contained in $\hm$, by ``sparse $\Gamma$ system" we will mean a sparse $\Gamma \cap \hm$ system.
\end{definition}

Clearly, if $\Gamma_0 \subseteq \Gamma_1$, then any sparse $\Gamma_0$ system is also a sparse $\Gamma_1$ system.  It is known that:
\begin{itemize}
    \item There is a sparse $C^\infty$ system (Bajpai-Velleman~\cite{MR3552748}), and in fact a sparse $D^\infty$ system (Cox-Elpers~\cite{Cox_Elpers}).
    \item The existence of a  sparse analytic system is equivalent to the Continuum Hypothesis (Cody-Cox-Lee~\cite{CodyCoxLee}).
\end{itemize}

It is natural to wonder about other classes of functions.  Following Mandelbrojt~\cite{MR0006354}, given a compact interval $K$ of the reals and a sequence $\{ M_k \}$ of positive real numbers, $\boldsymbol{ C_K \{ M_k \}}$ denotes the class of all $C^\infty$ functions $f: K \to \mathbb{R}$ such that there exist positive real numbers $\beta_f$ and $B_f$ such that 
\begin{equation}\label{eq_MainIneq}
\forall k \ge 0 \ \ \| f^{(k)}\| \le \beta_f B_f^k M_k. \tag{*}
\end{equation}
\noindent For functions with domain $\mathbb{R}$, Rudin~\cite{MR0924157} and Hughes~\cite{MR0272965} used $\boldsymbol{C \{ M_k \}}$ to denote the set of $C^\infty$ functions $f: \mathbb{R} \to \mathbb{R}$ such that there are positive $\beta_f$ and $B_f$ such that \eqref{eq_MainIneq} holds (on all of $\mathbb{R}$).  Observe that a function $f \in C \{ M_k \}$ must be bounded by the constant $\beta_f M_0$; so in particular $C \{ M_k \}$ is disjoint from $\text{Homeo}^+(\mathbb{R})$ and hence not appropriate for the study of sparse systems from Definition \ref{def_SparseSystem}.  So we will also consider the class $\boldsymbol{C_{\text{loc}} \{ M_k \}}$, which refers to the set of all $C^\infty$ functions $f:\mathbb{R} \to \mathbb{R}$ such that $f \restriction K \in C_K \{ M_k \}$ for every compact interval $K \subset \mathbb{R}$.  Observe that for any sequence $\{ M_k \}$ of positive reals,
\begin{equation}\label{eq_Chain}
\text{real polynomials} \subseteq C_{\text{loc}} \{ M_k \} \subseteq C^\infty(\mathbb{R}).
\end{equation}
\noindent For example, when $M_k = k!$, $C_{\text{loc}} \{ M_k \}$ is the class of real analytic functions, and $C \{ M_k \}$ is the class of bounded real analytic functions that can be extended to a holomorphic function on a strip about the real axis in $\mathbb{C}$ (see Rudin~\cite{MR0924157}, Theorem 19.9).

We shall refer to any class of the form $C_{\text{loc}} \{ M_k \}$ as a \textbf{locally Denjoy-Carleman class}.\footnote{This naming convention is in line with the current literature.  We could not locate who first introduced classes of the form $C_K \{ M_k \}$, though the Denjoy-Carleman theorem---which characterizes when such classes are quasi-analytic---answered a question that Hadamard had raised about such classes in 1912 (Mandelbrojt~\cite{MR0006354}).}  Our paper addresses the question:
\begin{question}\label{q_MainQuestion}
Which locally Denjoy-Carleman classes carry sparse systems?
\end{question}
We completely answer this question, but the answer depends on whether or not the Continuum Hypothesis (CH) holds:
\begin{enumerate}
    \item  CH implies that \textbf{every} locally Denjoy-Carleman class carries a sparse system.
    \item $\neg \text{CH}$ implies that a locally Denjoy-Carleman class carries a sparse system if and only if it is not \emph{quasi-analytic} (defined below). 
\end{enumerate}

These facts follow from Theorems \ref{thm_non_qa_sparse} and \ref{thm_CH} below.  We recall the definition of quasi-analyticity.  A class $\Gamma$ of infinitely differentiable functions  is called \textbf{quasi-analytic} if the following implication holds for every $f \in \Gamma$:  if, for some $x_0$, $f^{(k)}(x_0) = 0$ holds for all $k \ge 0$, then $f$ is identically zero.  Under the additional assumption that $\{M_k\}$ is log-convex, the Denjoy-Carleman Theorem~\cite{MR0924157} states that $C\{M_k\}$ is quasi-analytic if and only if $\sum_{k=1}^\infty M_{k-1}/M_{k} = \infty$.  We provide a new characterization of quasi-analyticity:

\begin{theorem}  \label{thm_Marco_char_quasi}
    $C\{M_k\}$ is a non-quasi-analytic class if and only if it contains an uncountable family $F$ such that $\left|\left\{ f(x) \ : \ f \in F \right\}\right|\le 2$ for all $x\in \mathbb R$.
\end{theorem}

The following theorem improves the Bajpai-Velleman~\cite{MR3552748} theorem that there is a sparse $C^\infty$ system.

\begin{theorem}\label{thm_non_qa_sparse}
All non-quasi-analytic locally Denjoy-Carleman classes carry sparse systems.    
\end{theorem}

Our other main theorem characterizes the Continuum Hypothesis in terms of which locally Denjoy-Carleman classes carry sparse systems:
\begin{theorem}\label{thm_CH}
The following are equivalent:
\begin{enumerate}
    \item\label{item_CH} The Continuum Hypothesis (CH)
    \item\label{item_EveryDenjoy-Carleman} Every locally Denjoy-Carleman class carries a sparse system.

    \item\label{item_AtLeastOneDenjoy-Carleman} At least one quasi-analytic locally Denjoy-Carleman class carries a sparse system.

    \item\label{item_SparseAnalytic} The particular quasi-analytic locally Denjoy-Carleman class $C_{\text{loc}} \{ k! \}$---i.e., the class of real analytic functions---carries a sparse system.
\end{enumerate}
\end{theorem}

\noindent The equivalence of \eqref{item_CH} with \eqref{item_SparseAnalytic} was proved in \cite{CodyCoxLee}.  The main new content here is the proof that \eqref{item_CH} implies \eqref{item_EveryDenjoy-Carleman} (and the proof that \eqref{item_AtLeastOneDenjoy-Carleman} implies \eqref{item_CH}, though this is a trivial modification of Erd\H{o}s~\cite{MR168482}).

$\Gamma$-sparse systems were introduced by Cody-Cox-Lee~\cite{CodyCoxLee} mainly because, by an argument ultimately due to Bajpai and Velleman~\cite{MR3552748}, the existence of such a system implies there is \textbf{no} good $\Gamma$-anonymous predictor for functions from $\mathbb{R}$ into $\mathbb{R}$ (see sections 4 and 5 of \cite{CodyCoxLee}).  So Theorems \ref{thm_non_qa_sparse} and \ref{thm_CH} yield the following corollaries, respectively: 

\begin{corollary}
If $\Gamma$ is a non-quasi-analytic locally Denjoy-Carleman class, then there is no good $\Gamma$-anonymous predictor for functions from $\mathbb{R}$ to $\mathbb{R}$.
\end{corollary}

\begin{corollary}
Assume CH.  For any locally Denjoy-Carleman class $\Gamma$, there is no good $\Gamma$-anonymous predictor for functions from $\mathbb{R}$ to $\mathbb{R}$.
\end{corollary}

Before continuing to the proofs of the above theorems, we observe that there is a kind of ``lower frontier" (in the subgroup lattice of $\text{Homeo}^+(\mathbb{R})$) on which classes can carry sparse systems:
\begin{lemma}\label{lem_polynomials}
    There is \textbf{no} sparse polynomial system.
\end{lemma}
\begin{proof}
Assume toward a contradiction that
\[
\left\{ f_P \ : \ P=(x_P,y_P) \in \mathbb{R}^2\right\}
\]
is a sparse polynomial system.  Consider the subcollection
\[
\mathcal{F}:=\left\{ f_{(0,y)} \ : \ y \in \mathbb{R} \right\}.
\]
By sparseness of the original system, 
\begin{equation}\label{eq_CtbleOutputsAtNonzeroX}
    \forall x \ne 0  \  \ Y_x:=\big\{ f_{(0,y)}(x) \ : \ y \in \mathbb{R} \big\} \text{ is countable.} 
\end{equation}

For each non-negative integer $n$, let $\mathcal{F}_n$ denote the set of members of $\mathcal{F}$ of degree $n$.  Since $\mathcal{F}$ is uncountable, there  exists some $n^*$ such that $\mathcal{F}_{n^*}$ is uncountable.

Since $\mathcal{F}_{n^*} \subseteq \mathcal{F}$, if we consider any value $x_1 \ne 0$, then
\[
\left\{ f(x_1) \ : \ f \in \mathcal{F}_{n^*} \right\}
\]
is contained in $Y_{x_1}$, which is countable.  So $\mathcal{F}_{n^*}$  can be partitioned into countably many subsets indexed by their outputs at $x_1$. Therefore, there exists some value $y_1$ such that 
\[
 G_1 := \left\{ f \in \mathcal{F}_{n^*} \ : \ f(x_1) = y_1 \right\}
\]
\noindent is uncountable.

Now choose any $x_2 \notin \{  0,x_1 \}$ and again using \eqref{eq_CtbleOutputsAtNonzeroX}, find a $y_2$ such that 
\[
G_2 := \left\{ f \in G_1 \ : \ f(x_2) = y_2 \right\}
\]
is uncountable. All members of $G_2$ have the same outputs at $x_1$ and $x_2$. If we continue this process with any distinct nonzero numbers $x_1, x_2, \dots, x_{n^*+1}$, then we obtain a set $G_{n^* + 1}$ of uncountably many polynomials of degree $n^*$, all passing through the  $n^*+1$ points 
\[
\big\{ (x_i,y_i) \ : \ 1 \le i \le n^* + 1 \big\}.
\]
This  violates the uniqueness of the Lagrange interpolating polynomial, yielding a contradiction. \end{proof}

We record the following basic fact, that to check membership of $f$ in $C_K\{ M_k \}$ for a compact interval $K$, it suffices to find $\beta_f$ and $B_f$ that work for all sufficiently large derivatives:
\begin{lemma}\label{lem_PastSomePointSuffices}
Suppose $K$ is a compact interval of $\mathbb{R}$, and $\{ M_k \}$ is a sequence of positive real numbers.  Suppose $N$ is a fixed natural number, $f \in C^\infty(K)$, and there exist $\beta_f$ and $B_f$ such that
\[
\forall k \ge N \ \ \| f^{(k)} \| \le \beta_f B_f^k M_k.
\]
Then $f \in C_K \{ M_k \}$.
\end{lemma}
\begin{proof}
 Since $K$ is compact and $f \in C^{\infty}(K)$,   
 \[
\beta_2:= \text{max} \left\{ \frac{ \| f^{(k)} \|}{ B_f^k M_k} \ : \ 0 \le k < N  \right\}
 \]
is finite.  Let $\widetilde{\beta}:= \text{max}(\beta_f, \beta_2)$.  Then $\| f^{(k)} \| \le \widetilde{\beta} B_f^k M_k$ for all $k \ge 0$.
\end{proof}

We also will use:

\begin{lemma}\label{lem_LogConvexification}
Given any (not necessarily log-convex) sequence $\{ M_k \}$ of positive reals:  $C\{ M_k \}$ is quasi-analytic if and only if $C_{\text{loc}} \{ M_k \}$ is quasi-analytic.
\end{lemma}
\begin{proof}
$C\{ M_k \} \subseteq C_{\text{loc}} \{ M_k \}$, so quasi-analyticity of $C_{\text{loc}} \{ M_k \}$  trivially implies quasi-analyticity of $C \{ M_k \}$. 

Now suppose $C_{\text{loc}} \{ M_k \}$ is not quasi-analytic, as witnessed by a nonzero $f \in C_{\text{loc}} \{ M_k \}$ such that (WLOG) $f^{(k)}(0) = 0$ for all $k \ge 0$.  Let $I$ be a sufficiently large closed interval including 0, so that $f \restriction I$ is a nonzero function.  Then $f \restriction I$ witnesses that $C_I \{ M_k \}$ is not quasi-analytic.  By Mandelbrojt~\cite{MR0006354} (see also Cohen~\cite{MR0225957}) it follows that the log-convexification $\{ M'_k \}$ of $\{ M_k \}$ has the property that $\sum_{k=1}^\infty M'_{k-1}/M'_k < \infty$.  Then by the version of the Denjoy-Carleman Theorem from Rudin~\cite{MR0924157}, $C\{ M_k' \}$   is not quasi-analytic. The lemma then follows since $M_k'\le M_k$ for all $k$.
\end{proof}

\section{New characterization of quasi-analyticity}

The following lemma is adapted to the quasi-analytic case from Erd\H{o}s~\cite{MR168482}.

\begin{lemma}\label{lem_erdos}
  Let $f_1,f_2,f_3$ be functions in a quasi-analytic class $C\{M_k\}$. Then the subset $S$ of all real numbers $x$ such that $|\{f_1(x),f_2(x),f_3(x)\}|<3$ is discrete. 
\end{lemma}

\begin{proof}
    Let $S_{12}$ be the set of all $x\in \mathbb R$ such that $f_1(x)=f_2(x)$. Suppose $S_{12}$ contains a convergent sequence $\{x_n\}$. Then, by repeated application of the Mean Value Theorem, we conclude that all derivatives of $f_1-f_2\in C\{M_k\}$ vanish at $\lim_n x_n$. Hence $f_1=f_2$. Repeating this argument twice we conclude that the similarly defined subsets $S_{13}$ and $S_{23}$ are also discrete. Hence $S=S_{12}\cup S_{13}\cup S_{23}$ is discrete.
\end{proof}

The following lemma can be thought of as a sharpening of a result attributed to Lyndon by Erd\H{o}s \cite{MR168482}.

\begin{lemma}\label{lem_lyndon}
If $C \{M_k\}$ is not quasi-analytic, then it contains an uncountable subset $F$ such that $\left|\left\{ f(x) \ : \ f \in F \right\}\right|\le 2$ for all $x\in \mathbb R$.
\end{lemma}

\begin{proof}
Let $\mathcal C\subseteq[0,1]$ be the standard Cantor set. By a theorem of Hughes~\cite{MR0272965}, there exists $g\in C\{M_k\}$ such that 
\begin{enumerate}
\item $g^{(n)}(x)=0$ for all $n\ge 0$ and for all $x\in \mathcal C\cup (1,\infty) \cup(-\infty,0)$;
\item $g(x)\neq 0$ for all $x \in (0,1) \setminus \mathcal{C}$.
\end{enumerate}
For each $a,b\in \mathcal C$ such that $a<b$, let $f_{ab}:\mathbb R\to \mathbb R$ be the function such that $f_{ab}(x)=g(x)$ if $x\in [a,b]$ and $f_{ab}(x)=0$ otherwise. Then the collection $F$ of all $f_{ab}$ as above satisfies the statement of the lemma.
\end{proof}

The proof of Theorem \ref{thm_Marco_char_quasi} follows easily by combining Lemma \ref{lem_erdos} and Lemma \ref{lem_lyndon}.

\section{Sparse systems for non-quasi-analytic classes}\label{sec_NonQA_sparse}

Our proof of Theorem \ref{thm_non_qa_sparse} follows the outline of the proof in section 4 of Bajpai-Velleman.  First, we need a lemma that yields extremely flat bump functions in $C \{ M_k \}$, provided the sequence $\{ M_k \}$ grows sufficiently fast.

%

\begin{lemma}\label{lem_ImproveHughes}
Assume $I$ is an open interval of real numbers, $\varepsilon > 0$, and $\{ M_k \}$ a sequence of positive real numbers such that
\begin{equation}\label{eq_SumSmallDivergeFast}
\sum_{k=1}^\infty M_{k-1}/M_k < \infty.
\end{equation}
Then there is a $C^\infty$ function $b$ such that $b(x) > 0$ for all $x \in I$, $b(x) = 0$ for all $x \in \mathbb{R} \setminus I$, and  
\[
 \forall k \ge 0 \ \  \| b^{(k)} \| \le \varepsilon^k M_k. 
\]
\end{lemma}

\begin{proof}
Define an auxilliary sequence
\[
L_n:= \begin{cases} 
      \frac{8 M_0}{\varepsilon} & n = 1 \\
      M_n & n \neq 1
   \end{cases}
\]
First we observe that, to prove the lemma, it suffices to find a $C^\infty$ function $b$ that is positive on $I$, zero on $\mathbb{R} \setminus I$, and with the property that for some positive constant $\beta$,
\begin{equation}\label{eq_aux}
\forall n \ \ || b^{(n)} || \le \beta \varepsilon^n L_n.
\end{equation}
Indeed, set $\alpha := \max \left\{1, \frac{8}{\varepsilon}\right\}$.  Then $L_n \le \alpha M_n$ for every $n$.  If we can manage to find a $C^\infty$ bump function $b$ with support $I$ satisfying \eqref{eq_aux}, then the function $\frac{1}{\beta \alpha} b$ will satisfy the conclusion of the lemma.

Let $E:= \mathbb{R} \setminus I$, which is a closed set of reals.  The assumption \eqref{eq_SumSmallDivergeFast} and definition of $\{ L_n \}$ ensures that
\[
\sum_{n=1}^\infty L_{n-1}/L_n < \infty.
\]
Hughes~\cite{MR0272965} demonstrates (with our $L_n$'s playing the role of his $M_n$'s) that, given any positive constant $\rho$, if we set
\[
\lambda_\rho:= \frac{L_0}{L_1} + \rho \ \underbrace{\sum_{n=2}^\infty \frac{L_{n-1}}{L_{n}}\left(\sum^\infty_{k=n} \frac{L_{k-1}}{L_k}\right)^{-\frac{1}{2}}}_{=:D}
\]
then there exists a function $f_\rho \in C^\infty(\mathbb{R})$ which is zero on $E$, and is positive on $E^\mathsf{c} (=I)$, such that $\forall n\in\mathbb{N}, ||f_\rho^{(n)}||_\infty\le 2(4\lambda_\rho)^n L_n$.  In particular, if we set
\[
\rho:=\frac{\varepsilon}{8D},
\]
then $\lambda_\rho=\frac{\varepsilon}{4}$.  Therefore, $||f_\rho^{(n)}||_\infty\le 2\varepsilon^n L_n$ for every $n \ge 0$, yielding the desired \eqref{eq_aux}.
\end{proof}

\begin{corollary}\label{cor_transition}
Suppose $I=(\ell,r)$, $\varepsilon > 0$, and $\{ M_k \}$ are as in the assumptions of Lemma \ref{lem_ImproveHughes}. Then there exists an infinitely differentiable function $s: [\ell,r] \to \mathbb{R}$ such that:
\begin{enumerate}
    \item $s$ is strictly increasing
    \item $s(\ell)=0$
    \item For all $k \ge 1$:  $s^{(k)}(\ell)=0=s^{(k)}(r)$ (where these are right and left derivatives, respecively)
    \item\label{item_s_norms} for all $k \ge 0$:  $\| s^{(k)} \| \le \varepsilon^k M_k$.
\end{enumerate}
\end{corollary}
\begin{proof}
    Let $b$ be as in the conclusion of Lemma \ref{lem_ImproveHughes}, and define $s_0$ on $[\ell,r]$ by
    \[
    s_0(x):=\int_0^x b(t)dt.
    \]
    Since (on $[\ell,r]$) $s^{(k)} = b^{(k-1)}$ for all $k \ge 1$, the properties of $b$ ensure that $s_0$ has all of the required properties, except possibly requirement \eqref{item_s_norms}.  Note that the properties of $b$ ensure that
    \begin{equation}
        \forall k \ge 1 \ \| s_0^{(k)} \| = \| b^{(k-1)} \| \le \varepsilon^{k-1} M_{k-1}.
    \end{equation}
Assumption \eqref{eq_SumSmallDivergeFast} implies that the terms $M_{k-1}/M_k$ converge to zero, so there is a natural number $N$ such that 
\[
\forall k \ge N \ \ \frac{M_{k-1}}{M_k} < \varepsilon = \frac{\varepsilon^{k}}{\varepsilon^{k-1}}
\]
and hence
\[
\forall k \ge N \ \ \| s_0^{(k)} \| \le \varepsilon^{k-1} M_{k-1} < \varepsilon^k M_k.
\]

Set
\[
A:= \text{max} \left\{ \frac{\| s_0^{(k)}  \| }{\varepsilon^k M_k}   \ : \ 0 \le k < N  \right\}
\]
If $A \le 1$ we set $s:= s_0$, and if $A > 1$, set $s:= \frac{1}{A} s_0$.  Then $s$ has the required properties.

\end{proof}

We are now ready to prove Theorem \ref{thm_non_qa_sparse}.  For the rest of this section, assume $C_{\text{loc}} \{ M_k \}$ is not quasi-analytic.  Then by Lemma \ref{lem_LogConvexification}, $C \{ M_k \}$ is not quasi-analytic.  So by Rudin~\cite{MR0924157} we may WLOG assume that $\{M_k\}$ is log-convex. Then, by the Denjoy-Carleman Theorem, 
\[
\sum_{k=1}^\infty \frac{M_{k-1}}{M_k} < \infty.
\]
For each rational $\Delta > 0$ and each positive integer $i$, fix a function 
\[
s_{\Delta,i}: [0,\Delta] \to \mathbb{R}
\]
satisfying the conclusion of Corollary \ref{cor_transition}, with $I = [0,\Delta]$ and $\varepsilon = 1/i$.  Define
\[
y(\Delta,i):= s_{\Delta,i}(\Delta).
\]
So, $s_{\Delta,i}$ has the following properties:
\begin{enumerate}
    \item It maps $[0,\Delta]$ onto $[0,y(\Delta,i)]$ in a strictly increasing fashion;
    \item All derivatives of $s_{\Delta,i}$ at the endpoints of the interval $[0,\Delta]$ are zero;
    \item For all $k \ge 0$, $\| s_{\Delta,i}^{(k)} \| \le (1/i)^k M_k$.
\end{enumerate}

For each pair $p,q$ of rational numbers, each pair $\Gamma$, $\Delta$ of positive rationals, and each positive integer $i$, define
\[
t_{p,q,\Delta, \Gamma,i}: [p,p+\Delta] \to \mathbb{R}
\]
by
\[
t_{p,q,\Delta,\Gamma, i}(x):=q + \frac{\Gamma}{y(\Delta,i)} s_{\Delta,i}(x-p)
\]
Notice that $t_{p,q,\Delta,\Gamma,i}$ maps the interval $[p,p+\Delta]$ onto $[q, q + \Gamma]$, and its derivatives at the endpoints are all zero.  

Let $\mathcal{T}$ denote the set of all such $t_{p,q,\Delta,\Gamma,i}$ (for $p,q,\Delta,\Gamma \in \mathbb{Q}$ and $i \in \mathbb{N}$) that have the property
\begin{equation}\label{eq_KeyRatio}
    \frac{\Gamma}{y(\Delta,i)} \le 1.
\end{equation}
Notice that $\mathcal{T}$ is a countable set of functions, and the inequality \eqref{eq_KeyRatio} ensures that 
\begin{equation}
   t_{p,q,\Delta,\Gamma,i} \in \mathcal{T} \ \implies \ \  \forall k \ge 1 \ \| t^{(k)}_{p,q,\Delta,\Gamma,i} \| \le \| s_{\Delta,i}^{(k)} \| \le (1/i)^k M_k. 
\end{equation}

Fix a point $P=(x_P,y_P)$ in the real plane; we need to define the function $h_P$ that will be part of our sparse $C_{\text{loc}} \{ M_k \}$ system.  

Fix any increasing sequence $\{ p_i \}$ of rational numbers converging to $x_P$, and set $\Delta_i:= p_{i+1}-p_i$.  Recursively define an increasing sequence $\{ q_i \}$ of rational numbers converging to $y_P$ such that
\begin{equation}
    \forall i \ \  y(\Delta_i,i) >  y_P-q_i  > 0.
\end{equation}
Let $\Gamma_i:=q_{i+1}-q_i$.  It follows that 
\begin{equation}
    \forall i \ \ \Gamma_i = q_{i+1}-q_i < y_P-q_i  < y(\Delta_i,i) 
\end{equation}
and hence $\displaystyle\frac{\Gamma_i}{y(\Delta_i,i)} \le 1$ for all $i$.  So 
\begin{equation}\label{eq_ith_piece}
\forall i \ge 0 \left(  t_{p_i,q_i,\Delta_i,\Gamma_i,i} \in \mathcal{T}  \text{ and } \forall k \ge 1 \left( \ \| t^{(k)}_{p_i,q_i,\Delta_i,\Gamma_i,i} \| \le (1/i)^k M_k \right)  \right).
\end{equation}
Notice that for each $i$, $t_{p_i,q_i,\Delta_i,\Gamma_i,i}$ meets $t_{p_{i+1},q_{i+1},\Delta_{i+1},\Gamma_{i+1},{i+1}}$ at the point $(p_{i+1},q_{i+1})$, and both have all vanishing derivatives at $p_{i+1}$.  So their union is a strictly increasing $C^\infty$ function mapping $[p_i,p_{i+2}]$ onto $[q_i,q_{i+2}]$.  Furthermore, since $(p_i,q_i)$ converges to $P=(x_P,y_P)$, the function $h: [p_0,x_P] \to [q_0,y_P]$ defined (as a set of ordered pairs) by 
\[
h:= \big\{ (x_P,y_P) \big\} \cup \bigcup_{i \in \mathbb{N}} t_{p_i,q_i,\Delta_i,\Gamma_i,i}
\]
is a strictly increasing function that is continuous on $[p_0,x_P]$ and infinitely differentiable on $[p_0,x_P)$.  To see that it is infinitely differentiable at $x_P$ too, fix a $k \ge 1$.  By \eqref{eq_ith_piece}, we have
\[
\lim_{i \to \infty} \| \underbrace{t^{(k)}_{p_i,q_i,\Delta_i,\Gamma_i,i}}_{h^{(k)} \restriction [p_i,p_{i+1}]} \| =0,
\]
so $\lim_{z \nearrow x_P} h^{(k)}(z)  = 0$.  By Lemma 6 of Bajpai-Velleman~\cite{MR3552748}, the $k$-th left derivative of $h$ at $x_P$ exists and is zero.  

Now \eqref{eq_ith_piece} also ensures that 
\begin{equation}
    \forall k \ge 1 \  \| h^{(k)} \| \le M_k, 
\end{equation}
and so by Lemma \ref{lem_PastSomePointSuffices}, $h \in C_{[p_0,x_P]} \{ M_k \}$.

This completes the heart of the construction, but the function we've constructed has domain $[p_0,x_P]$, not $\mathbb{R}$.  But still using members of $\mathcal{T}$, we can easily extend $h$ to a strictly increasing $h_L: (-\infty,x_P] \to (-\infty,y_P]$ such that for all $k \ge 1$, $\| h_L^{(k)} \| \le M_k$, as follows.  For $n \in \mathbb{N}$ set $u_n:= p_0- n$, $v_n := q_0 - n \cdot y(1,1)$, and use the function $t_{u_{n+1}, v_{n+1}, 1, y(1,1),1 }$ to map $[u_{n+1}, u_n]$ onto $[v_{n+1},v_n]$.  Combining with the earlier function on $[p_0,x_P]$, this yields a function mapping $(-\infty,x_P]$ onto $(-\infty,y_P]$ with the desired properties.  And what we've done so far can obviously be done ``from the right" of the point $P=(x_P,y_P)$ as well, yielding a function with similar properties on $[x_P,\infty)$.  Thus we obtain an increasing bijection $h_P: \mathbb{R} \to \mathbb{R}$ passing through the point $P=(x_P,y_P)$, such that 
\[
\forall k \ge 1 \ \| h_P^{(k)} \| \le M_k.
\]
In particular, $h_P \in C_{\text{loc}} \{ M_k \}$.   

Finally, we verify the sparseness of the system $\big\langle h_P \ : \ P = (x_P,y_P) \in \mathbb{R}^2 \big\rangle$.  Fix any $u \in \mathbb{R}$, and consider the collection of all points $P$ such that $u \ne x_P$.  By construction, for any such $P$, $h_P(u) = t(u)$ for some $t$ in the countable set $\mathcal{T}$; so 
\[
\left\{ h_P(u) \ : \ P \in \mathbb{R}^2 \text{ and } u \ne x_P \right\} \subseteq \left\{\vphantom{t^{-1}(u)} t(u) \ : \ t \in \mathcal{T} \right\}
\]
and the right side is countable, since $\mathcal{T}$ is countable.  Similarly, 
\[
\left\{ h^{-1}_P(u) \ : \ P \in \mathbb{R}^2 \text{ and } u \ne y_P \right\} \subseteq \left\{ t^{-1}(u) \ : \ t \in \mathcal{T} \right\}
\]
is countable.

\section{CH and sparse systems for all locally Denjoy-Carleman classes}

We prove Theorem \ref{thm_CH}.  Equivalence of \eqref{item_CH} with \eqref{item_SparseAnalytic} was proved in \cite{CodyCoxLee}.   \eqref{item_EveryDenjoy-Carleman} trivially implies \eqref{item_AtLeastOneDenjoy-Carleman}, and \eqref{item_SparseAnalytic} trivially implies \eqref{item_AtLeastOneDenjoy-Carleman}.  It remains to prove that \eqref{item_AtLeastOneDenjoy-Carleman} implies \eqref{item_CH} and that \eqref{item_CH} implies \eqref{item_EveryDenjoy-Carleman}.

To see that \eqref{item_AtLeastOneDenjoy-Carleman} implies \eqref{item_CH}, suppose $\left\{ f_P \ : \ P \in \mathbb{R}^2  \right\}$ is a sparse $C_{\text{loc}} \{ M_k \}$ system, where $C_{\text{loc}} \{ M_k \}$ is quasi-analytic.  Then 
\[
\left\{ f_{(0,y)} \ : \ y \in \mathbb{R} \right\}
\]
is an uncountable family of (restrictions of) functions from $C_{\text{loc}} \{ M_k \}$, all with domain $(-\infty,0)$, which by sparseness of the original system has what Erd\H{o}s~\cite{MR168482} calls ``property $P_0$".  The quasi-analyticity ensures that if $f$ and $g$ are distinct members of $C_{\text{loc}} \{ M_k \}$, then $\left\{ x \in \mathbb{R} \ : \ f(x) = g(x) \right\}$ is discrete, hence countable.  The failure of CH then follows by exactly the same arguments as Erd\H{o}s~\cite{MR168482}.

Now we prove \eqref{item_CH} implies \eqref{item_EveryDenjoy-Carleman}. We first prove the following ZFC theorem, which strengthens Theorem 5 of \cite{CodyCoxLee}:

\begin{theorem}\label{thm_Interpolate}
Suppose $\mathcal{E}$ is a partition of $\mathbb{R}$ into dense subsets of $\mathbb{R}$; for each $z \in \mathbb{R}$, let $E_z$ denote the unique member of $\mathcal{E}$ containing $z$.

Then for any $P=(x_P,y_P) \in \mathbb{R}^2$, any countable set $W$ of reals, and any sequence $\{ N_k \}$ of positive reals, there is an entire $f: \mathbb{C} \to \mathbb{C}$ such that:
\begin{enumerate}
    \item $f \restriction \mathbb{R}$ is an increasing bijection from the reals to the reals with strictly positive derivative;
    \item $f(x_P)=y_P$;
    \item for each $w \in W$:
    \begin{itemize}
        \item if 
         $w \ne x_P$ then $f(w) \in E_w$;
        \item if $w \ne y_P$ then $f^{-1}(w) \in E_w$; and
    \end{itemize}
    \item\label{item_NewClause} $f\restriction \mathbb{R}  \in C_{\text{loc}} \{ N_k \}$.
\end{enumerate}
\end{theorem}
\begin{proof}
We slightly modify the proof of Theorem 5 from \cite{CodyCoxLee}.  In that proof, the desired $f$ was the limit of a recursively-constructed sequence $\{ f_n \}$ of polynomials, where $f_0(z) =\frac{3}{2}(z-x_P) + y_P$, and for each $n \ge 1$, $f_n$ was of the form
\[
f_n = f_{n-1} + \alpha_n M_n h_n
\]
for some (recursively-defined) polynomial $h_n$, some positive real $\alpha_n$, and some $M_n \in [0,1]$, defined in that order.  The proof maintains a list of 7 inductive clauses, labelled (I)$_n$ through (VII)$_n$, which ensure that the sequence $\{ f_n \}$ was uniformly Cauchy on all compact $D \subset \mathbb{C}$, and that the limit\footnote{Which is uniform on every compact $D \subset \mathbb{C}$.} of the $f_n$'s satisfies all the properties listed in Theorem \ref{thm_Interpolate}, except possibly our new clause \ref{item_NewClause}.  

We will modify the recursion from \cite{CodyCoxLee} ever so slightly; essentially at stage $n$ of the recursion, we define $h_n$ exactly as in \cite{CodyCoxLee}, but then shrink $\alpha_n$ even further; this has no effect on the maintenance of the inductive clauses (I)$_n$ through (VII)$_n$.  We also recursively construct another auxilliary sequence $\{ B_n \}$ of (typically large) positive real numbers, and add another inductive clause to the seven already listed in \cite{CodyCoxLee}:
\[
\text{(VIII)}_{\textcolor{red}{n}}: \ \ \ \forall i \ \forall j \ \ \left(  i \le j \le \textcolor{red}{n} \ \implies \  \forall k \ge 1 \ \    \left\lVert   f_j^{(k)} \restriction D_i \right\rVert  < B_i^k N_k \right)
\]
\noindent Here, $D_i$ denotes the closed disc in $\mathbb{C}$ of radius $i$ centered at the origin.  Note that since each $f_j$ is a polynomial, the inequality $\left\lVert   f_j^{(k)} \restriction D_i \right\rVert  < B_i^k N_k$ trivially holds for $k \ge \text{degree}(f_j)$, so the content of (VIII)$_n$ is really about a finite collection of values of $k$.

Suppose $\alpha_{n-1}$, $M_{n-1}$, $h_{n-1}$, and $B_{n-1}$ have been defined, and that (I)$_{n-1}$ through (VII)$_{n-1}$ hold, along with our new inductive clause
\[
\text{(VIII)}_{\textcolor{red}{n-1}}: \ \ \ \forall i \ \forall j \ \ \left(  i \le j \le \textcolor{red}{n-1} \ \implies \  \forall k \ge 1 \ \    \left\lVert   f_j^{(k)} \restriction D_i \right\rVert  < B_i^k N_k \right).
\]

The inductive step of the proof is almost verbatim from \cite{CodyCoxLee}, with the following alterations.  First, the polynomial $h_n$ is defined exactly as in \cite{CodyCoxLee} (top of page 5).  They then argue that for all sufficiently small choices of $\alpha_n > 0$, there is an $M_n \in [0,1]$ such that the polynomial
\[
f_n := f_{n-1} + \alpha_n M_n h_n
\]
satisfies their inductive clauses (I)$_n$ through (VII)$_n$.  Using assumption (VIII)$_{n-1}$, we simply shrink $\alpha_n$ further, if necessary, to ensure that the following finitely many additional constraints are met:
\begin{equation}\label{eq_before_defining_M_n}
    \forall i \le n -1  \ \ \forall k \le \text{degree} \big(  f_{n-1} + \alpha_n h_n \big) : \  \big\lVert f_{n-1} \restriction D_i \big\rVert + \big\lVert \alpha_n  h_n \restriction D_i  \big\rVert < B_i^k N_k.
\end{equation}

Then define $M_n \in [0,1]$ (based on choice of $\alpha_n$) exactly as in \cite{CodyCoxLee}.  Then by \eqref{eq_before_defining_M_n} and the fact that $M_n \in [0,1]$, it follows that
\begin{align*}
    \forall i \le n -1  \ \ \forall k \le \text{degree} \left(  \overbrace{f_{n-1} + \alpha_n M_n h_n}^{f_n:=} \right) : \
    \begin{split}\big\lVert f_n \restriction D_i \big\rVert &\le  \big\lVert f_{n-1} \restriction D_i \big\rVert + \big\lVert \alpha_n M_n h_n \restriction D_i  \big\rVert\\ &< B_i^k N_k.\end{split}
\end{align*}

Hence, $f_n$ satisfies the relevant requirements for all $i < n$.  To satisfy the requirement at $n$, simply define $B_n$ to be any strict upper bound of the set
\begin{equation}
\Big\{ \big\lVert f_n^{(k)} \restriction D_n \big\rVert \ : \ k \le \text{degree}(f_n)  \Big\}.
\end{equation}

Thus, we have achieved (VIII)$_n$.

Since the sequence of $f_n$'s satisfy clauses (I)$_n$ through (VII)$_n$ of \cite{CodyCoxLee}, by that theorem their limit $f$ satisfies all clauses of Theorem \ref{thm_Interpolate} listed before clause \ref{item_NewClause}.  And, the fact that our additional clause (VIII)$_n$ holds for all $n$ ensures that
\begin{equation}
\forall i \in \mathbb{N} \ \forall k \ge 1 \ \big\lVert f^{(k)} \restriction D_i  \big\rVert \le B_i^k N_k
\end{equation}
This implies that $f \restriction D_i \in C_{D_i} \{ N_k \}$ for all $i$, and hence $f \in C_{\text{loc}} \{ N_k \}$.

\end{proof}

Now let $\{ N_k \}$ be any sequence of positive reals.  The existence of a sparse $C_{\text{loc}} \{ N_k \}$ system under CH follows, \emph{mutatis mutandis}, as in the proof of Theorem 2 of \cite{CodyCoxLee} (with ``analytic" replaced by ``$C_{\text{loc}} \{ N_k \}$", and Theorem 5 of \cite{CodyCoxLee} replaced by our Theorem \ref{thm_Interpolate}). 
\section{Concluding remarks and open questions}

We summarize the current knowledge of sparse systems, and present some open questions.  We know:
\begin{itemize}[leftmargin=11mm, itemindent=0pt, labelwidth=4mm, rightmargin=7mm]
  \item Classes of real polynomials \textbf{never} carry sparse systems (Lemma \ref{lem_polynomials}).  
  
  \item Non-quasi-analytic locally Denjoy-Carleman classes \textbf{always} carry sparse systems (Theorem \ref{thm_non_qa_sparse}).
    \begin{itemize}[leftmargin=3mm]
  \item Assuming CH, \textbf{all} locally Denjoy-Carleman classes carry sparse systems (Theorem \ref{thm_CH});

  \item Assuming $\neg \text{CH}$, the locally Denjoy-Carleman classes that carry sparse systems are exactly the non-quasi-analytic ones (follows from Theorems \ref{thm_CH} and \ref{thm_non_qa_sparse}).
    \end{itemize}
\end{itemize}

This completely answers Question \ref{q_MainQuestion}, and establishes some clear dividing lines between classes that carry sparse systems and those that don't:
\begin{itemize}[rightmargin=7mm]
  \item Under CH, there is a divide between classes of polynomials and locally Denjoy-Carleman classes; all of the latter carry sparse systems, while none of the former do.  
  
  \item If CH fails, there is a divide between quasi-analytic locally Denjoy-Carleman classes and the non-quasi-analytic ones; all of the latter carry sparse systems, while none of the former do.   

\end{itemize}

What about other intermediate classes?  Regarding the CH setting, there is no obvious natural candidate for classes between polynomials and locally Denjoy-Carleman.  If we were to only require the sequence $\{ M_k \}$ to be non-negative in the definition of the locally Denjoy-Carleman class $C_{\text{loc}} \{ M_k \}$, then if $M_k = 0$ for at least one $k$, $C_{\text{loc}} \{ M_k \}$ would be contained in the polynomials and hence, by Lemma \ref{lem_polynomials}, could not carry a sparse system.

\end{document}